\documentclass[12pt]{amsart}
\usepackage{amssymb,amsmath,amscd}
\usepackage{graphicx}
\usepackage{mathrsfs}
\newtheorem{thm}{\bf Theorem}[section]
\newtheorem{prop}[thm]{\sc Proposition}
\newtheorem{lem}[thm]{\sc Lemma}
\newtheorem{cor}[thm]{\sc Corollary}
\theoremstyle{definition}
\theoremstyle{definition}\newtheorem{de}[thm]{\sc Definition}
\theoremstyle{definition}
\theoremstyle{definition}
\theoremstyle{definition}
\theoremstyle{definition}

\numberwithin{equation}{section}

\DeclareMathOperator{\supp}{supp}

\DeclareMathOperator{\N}{\mathbb N}

\DeclareMathOperator{\R}{\mathbb R}

\begin{document}
\title[sample]{Subelliptic Operators on Weighted Folland-Stein Spaces}
\author{Hung-Lin Chiu}
\address{Department of Mathematics, National Tsing Hua University, Hsinchu, Taiwan 300, R.O.C.}
\email{hlchiu@math.nthu.edu.tw}
\subjclass{1991 Mathematics Subject Classification.\ 32V05, 32V20}
\keywords{Key Words: Heisenberg group, Asymptotically flat pseudo-hermitian manifolds, Weighted Folland-Stein spaces, Spherical CR structure}

\begin{abstract}
In this paper, we show that the sub-Laplacian of an asymptotically flat pseudo-hermitian manifold defined on a suitable weighted Folland-Stein spaces is an isomorphism. Similarly in dimension $5$ a certain Dirac-type operator is shown to be an isomorphism, which is used to resolve the CR positive mass problem
\end{abstract}

\maketitle

\section{Introduction}
The CR positive mass problem plays an essential role in CR geometry. As well known, we need a CR positive mass theorem to solve the CR Yamabe problem for the cases which either the CR dimension $n=1$ or the CR manifold $M$ is spherical with higher CR dimension. When $n=1$, this was shown by Cheng, Malchiodi and Yang \cite{CMY}. On the other hand, when $n\geq 2$ and $M$ is spherical, this was finished by Cheng, Yang and the author \cite{CCY} through showing that the developing map is injective. However in the case $n=2$, we need an extra condition on the growth rate of the Green's function on the universal cover of $M$. So in the case $n=2$, the CR positive mass theorem is not really completed. In the paper \cite{CC}, Cheng and the author showed that for $n=2$, $M$ being spherical, if moreover $M$ has a spin structure, then we have the CR positive mass theorem built up through a spinorial approach.

Recall that in 1982, E. Witten described a proof of the positive mass theorem using spinors (see  \cite{W,PT}). Applying a Weitzenbock-type formula to $\psi$ satisfying $D^{2}\psi=0$ and approaching a constant spinor at infinity and integrating after taking the inner product with $\psi$, we then pick up the $p$-mass from the boundary integral and obtain a Witten-type formula for the $p$-mass. So the non-negativity of $p$-mass follows. Therefore, for the CR positive mass theorem, it suffices to show that the square of the Dirac operator $D^2$ on some suitable weighted Folland-Stein spaces is an isomorphism. In this paper we prove that both sublaplacian $\Delta_{b}$ (see Theorem \ref{mainres02}) and $D^{2}$ (see Theorem \ref{mainres03}) on suitable spaces are isomorphisms.\\

{\bf Acknowledgement:} The author would like to thank the Ministry of Science and Technology of Taiwan for the support by MOST 109-2115-M-007-004 -MY3.
We would also like to thank Professor Jih-Hsin Cheng for all suggestions and discussions on this paper.

\section{Operations on Weighted Folland-Stein Spaces}
For basic material about the Heisenberg group, we refer the reader to Section $7$ Appendix in \cite{CC}. Let $H_{n}$ be the Heisenberg group and $Q=2n+2$ be the homogeneous dimension of $H_{n}$. The summation convention applies and derivatives along $\mathring{e}_{a}, 1\leq a\leq 2n$, may be denoted by 
subscripts $u_{a}=\mathring{e}_{a}u$. The sub-gradient $\mathring{\nabla}_{b}u=\sum_{a=1}^{2n}u_{a}\mathring{e}_{a}$. The sub-laplacian $\mathring{\Delta}_{b}u=\sum_{a=1}^{2n}(\mathring{e}_{a})^{2}u$.
Set
\begin{equation}\label{opw01}
\rho=(|z|^{4}+t^{2})^{\frac{1}{2}},\ \ \sigma=(1+\rho^{4})^{\frac{1}{4}},
\end{equation}
Denote $B_{R}=B_{R}(0)$ as the closed Heisenberg ball of radius $R$ and center $0$. The annulus is $A_{R}=B_{2R}\setminus B_{R}$, and the exterior domain is 
$E_{R}=H_{n}\setminus B_{R}$.

On $H_{n}$, we have the dilation operator defined by
\begin{equation}\label{opw02}
R(z,t)=(Rz,R^{2}t)
\end{equation}

\subsection{The weighted Folland-Stein spaces}

For $\delta\in \R$, we define the weighted Lebesgue spaces $L^{p}_{\delta},L'^{p}_{\delta}$ with weight $\delta\in\R$ as the spaces of measurable functions in $L^{p}_{loc}(H_{n}),L^{p}_{loc}(H_{n}\setminus\{0\})$,
respectively, for which the norms defined by
\begin{equation}\label{opw03}
\|u\|_{p,\delta}=\left\{\begin{array}{ll}\left(\int_{H_{n}}|\sigma^{-\delta}u|^{p}\sigma^{-Q}dV_{(z,t)}\right)^{\frac{1}{p}},&\ \ p<\infty,\\
\ \ \textrm{ess sup}_{H_{n}}(\sigma^{-\delta}|u|),&\ \ p=\infty,
\end{array}\right.
\end{equation}
and
\begin{equation}\label{opw04}
\|u\|'_{p,\delta}=\left\{\begin{array}{ll}\left(\int_{H_{n}\setminus\{0\}}|\rho^{-\delta}u|^{p}\rho^{-Q}dV_{(z,t)}\right)^{\frac{1}{p}},&\ \ p<\infty,\\
\ \ \textrm{ess sup}_{H_{n}\setminus\{0\}}(\rho^{-\delta}|u|),&\ \ p=\infty,
\end{array}\right.
\end{equation}
are finite, where $dV_{(z,t)}=\Theta\wedge(d\Theta)^{n}$ and $\Theta=dt+\frac{1}{2}\sum_{\beta=1}^{n}(iz^{\beta}d\bar{z}^{\beta}-i\bar{z}^{\beta}dz^{\beta})$.

For $k\in\N\cup\{0\}$, the weighted Folland-Stein spaces $S^{p}_{k,\delta},S'^{p}_{k,\delta}$ are the spaces of $u$ such that
\begin{equation}
|\nabla^{j}u|\in L^{p}_{\delta-j},L'^{p}_{\delta-j},\ \ \textrm{for}\ 0\leq j\leq k,
\end{equation}
respectively, with norms
\begin{equation}\label{opw05}
\|u\|_{k,p,\delta}=\sum_{j=0}^{k}\|\nabla^{j}u\|_{p,\delta-j},
\end{equation}
\begin{equation}\label{opw06}
\|u\|'_{k,p,\delta}=\sum_{j=0}^{k}\|\nabla^{j}u\|'_{p,\delta-j}
\end{equation}
where $\nabla, \nabla^{2}, \cdots$ denote the subdradient, subhession, $\cdots$ for simplcity.

Observe that $C^{\infty}_{0}(H_{n}),C^{\infty}_{0}(H_{n}\setminus\{0\})$ are dense in $S^{p}_{k,\delta},S'^{p}_{k,\delta}$, respectively, for $\delta<0$ and $1\leq p<\infty$ and that 
$L^{p}_{\delta}=L^{p}(H_{n})$ for $\delta=-Q/p$.

\subsection{Global weighted inequalities}
As R. Bartnik pointed out in \cite{Bar} that the indexing chosen for the weights has the advantage that it directly describes the growth at infinity (see \eqref{opw15} below). In addition, the rescaled function, defined by
\begin{equation}\label{opw07}
u_{R}(z,t)=u(R(z,t)),
\end{equation}
will satisfy the estimates \eqref{opw08} and \eqref{opw09} which follow by a simple change of variables. We have

\begin{equation}\label{opw08}
\|u_{R}\|'_{k,p,\delta}=R^{\delta}\|u\|'_{k,p,\delta}.
\end{equation}
Similarly, with an obvious notation for norms over subsets of $H_{n}$, we have 
\begin{equation}\label{opw09}
\|u\|_{k,p,\delta;A_{R}}\approx R^{-\delta}\|u_{R}\|_{k,p,\delta;A_{1}},
\end{equation}
for $R\geq 1$, where "$\approx$" means "is comparable to", independent of $R\geq 1$. The estimate \eqref{opw09} is the key to proving global weighted inequalities from local inequalities.

\begin{prop}\label{ebe1}
If $1\leq p\leq q\leq\infty,\ \delta_{2}<\delta_{1}$ and $u\in L^{q}_{\delta_{2}}$, then 
\begin{equation}\label{opw10}
\|u\|_{p,\delta_{1}}\leq c\|u\|_{q,\delta_{2}},
\end{equation}
and hence $L^{q}_{\delta_{2}}\subset L^{p}_{\delta_{1}}$.
\end{prop}
\begin{proof} The estimate (\ref{opw10}) follows directly from the definition and H\"{o}lder inequality.
If $p=q<\infty$, we write
\begin{equation*}
\frac{|u|^{p}}{\sigma^{\delta_{1}p+Q}}=\frac{|u|^{q}}{\sigma^{\delta_{2}q+Q}}\ \frac{1}{\sigma^{(\delta_{1}-\delta_{2})q}}.
\end{equation*}
Then (\ref{opw10}) holds in this case. If $p<q<\infty$, we write
\begin{equation*}
\frac{|u|^{p}}{\sigma^{\delta_{1}p+Q}}=\frac{|u|^{p}}{\sigma^{\delta_{2}p+\frac{Qp}{q}}}\ \frac{1}{\sigma^{(\delta_{1}-\delta_{2})p}+\frac{Q(q-p)}{q}}.
\end{equation*}
Then by H\"{o}lder's inequality
\begin{equation*}
\begin{split}
\int_{E_{R}}\frac{|u|^{p}}{\sigma^{\delta_{1}p+Q}}\ dV&\leq\left[\int_{E_{R}}\left(\frac{|u|^{p}}{\sigma^{\delta_{2}p+\frac{Qp}{q}}}\right)^{q/p}\right]^{p/q}\left[\int_{E_{R}}\left(\frac{1}{\sigma^{(\delta_{1}-\delta_{2})p}+\frac{Q(q-p)}{q}}\right)^{q/(q-p)}\right]^{(q-p)/q}\\
&=\left[\int_{E_{R}}\left(\frac{|u|^{q}}{\sigma^{\delta_{2}q+Q}}\right)dV\right]^{p/q}\left[\int_{E_{R}}\left(\frac{1}{\sigma^{(\delta_{1}-\delta_{2})\frac{pq}{q-p}+Q}}\right)dV\right]^{(q-p)/q},
\end{split}
\end{equation*}
where the second term is integrable on $E_{R}$. Therefore (\ref{opw10}) holds. Finally, suppose that $p<q=\infty$, then
\begin{equation*}
\frac{|u|^{p}}{\sigma^{\delta_{1}p+Q}}=\frac{|u|^{p}}{\sigma^{\delta_{2}p}}\frac{1}{\sigma^{(\delta_{1}-\delta_{2})p+Q}}
\leq \|u\|^{p}_{\infty,\delta_{2}}\ \frac{1}{\sigma^{(\delta_{1}-\delta_{2})p+Q}},\ a.e.
\end{equation*}
Again that $\frac{1}{\sigma^{(\delta_{1}-\delta_{2})p+Q}}$ is integrable on $H_{n}$, we get (\ref{opw10}). Suppose $p=q=\infty$. We fix a point, since $\sigma\geq 1$, we have that $(\sigma^{-1})^{x}$ is a decreasing function, hence
$(\sigma^{-1})^{\delta_{1}}\leq(\sigma^{-1})^{\delta_{2}}$. And it is easy to see that if
\[m\{(z,t)\ |\ \sigma^{-\delta_{2}}|u(z,t)|>M\}=0\]
then
\[m\{(z,t)\ |\ \sigma^{-\delta_{1}}|u(z,t)|>M\}=0.\]
Therefore
\begin{equation*}
\begin{split}
\|u\|_{\infty,\delta_{1}}&=ess\ sup(\sigma^{-\delta_{1}}|u|)\\
&=\inf\{M\ |\ m\{(z,t)\ |\ \sigma^{-\delta_{1}}|u(z,t)|>M\}=0\}\\
&\leq \inf\{M\ |\ m\{(z,t)\ |\ \sigma^{-\delta_{2}}|u(z,t)|>M\}=0\}\\
&=ess\ sup(\sigma^{-\delta_{2}}|u|)=\|u\|_{\infty,\delta_{2}},
\end{split}
\end{equation*}
which gives \eqref{opw10}. Finally, suppose that $p<q=\infty$, since
\begin{equation*}
\begin{split}
|\sigma^{-\delta_{1}}u^{p}|\sigma^{-Q}&=|\sigma^{-\frac{\delta_{2}}{p}}u\sigma^{-\delta_{1}+\frac{\delta_{2}}{p}}|^{p}\sigma^{-Q}\\
&=\sigma^{-\delta_{2}}|u|\sigma^{-Q-p\delta_{1}+\delta_{2}}\\
&\leq M\frac{1}{\sigma^{Q+(p\delta_{1}-\delta_{2})}},\ \ \textrm{almost everywhere},
\end{split}
\end{equation*}
we have
\begin{equation*}
\begin{split}
\int_{H_{n}}\frac{|u|^{p}}{\sigma^{\delta_{1}p+Q}}dV_{(z,t)}&\leq\int_{H_{n}}\frac{M}{\sigma^{Q+(p\delta_{1}-\delta_{2})}}dV_{(z,t)}\\
&\leq C\ ess\ sup(\sigma^{-\delta_{2}}|u|).
\end{split}
\end{equation*}
We thus complete the proof.
\end{proof}

\begin{prop}[{\bf H\"{o}lder inequality}]
If $u\in L^{q}_{\delta_{1}}, v\in L^{r}_{\delta_{2}}$ and $\delta=\delta_{1}+\delta_{2}, \frac{1}{p}=\frac{1}{q}+\frac{1}{r}, 1\leq p,q,r\leq\infty$, then 
\begin{equation}\label{opw11}
\|uv\|_{p,\delta}\leq \|u\|_{q,\delta_{1}}\|v\|_{r,\delta_{2}},
\end{equation}
\end{prop}
\begin{proof} The estimate (\ref{opw11}) follows directly from the definition and H\"{o}lder inequality.
If $1< p,q,r<\infty$, then we write
\begin{equation*}
\frac{|uv|^{p}}{\sigma^{\delta p+Q}}=\frac{|u|^{p}}{\sigma^{\delta_{1}p+\frac{Qp}{q}}}\ \frac{|v|^{p}}{\sigma^{\delta_{2}p+\frac{Qp}{r}}}.\end{equation*}
By H\"{o}lder's inequality, we have
\begin{equation*}
\int_{H_{n}}\frac{|uv|^{p}}{\sigma^{\delta p+Q}}dV=\left(\int_{H_{n}}\left(\frac{|u|^{p}}{\sigma^{\delta_{1}p+\frac{Qp}{q}}}\right)^{q/p}dV\right)^{p/q}\left(\int_{H_{n}}\left(\frac{|v|^{p}}{\sigma^{\delta_{2}p+\frac{Qp}{r}}}\right)^{r/p}dV\right)^{p/r}.
\end{equation*}
This shows (\ref{opw11}). If $p=q<\infty,\ r=\infty$, then 
\begin{equation*}
\frac{|uv|^{p}}{\sigma^{\delta p+Q}}=\frac{|u|^{p}}{\sigma^{\delta_{1} p+Q}}\ \frac{|v|^{p}}{\sigma^{\delta_{2} p}}\leq\frac{|u|^{p}}{\sigma^{\delta_{1} p+Q}}\ \|v\|^{p}_{\infty,\delta_{2}},\ a.e.
\end{equation*}
This shows (\ref{opw11}). The case $p=q<\infty,\ r=\infty$ is similar. In the end, suppose that $p=q=r=\infty$. Let $M_{1}, M_{2}$ be two numbers such that 
\[m\{(z,t)\ |\ \sigma^{-\delta_{1}}|u(z,t)|>M_{1}\}=0,\ m\{(z,t)\ |\ \sigma^{-\delta_{2}}|v(z,t)|>M_{2}\}=0.\]
Then we have
\[\sigma^{-\delta}|uv|=\sigma^{-\delta_{1}}|u|\cdot\sigma^{-\delta_{2}}|v|\leq M_{1}M_{2},\ \ \textrm{almost everywhere},\]
that is,
\[m\{(z,t)\ |\ \sigma^{-\delta}|uv|>M_{1}M_{2}\}=0,\]
which implies
\begin{equation}
\begin{split}
\|uv\|_{\infty,\delta}&=ess\ sup(\sigma^{-\delta}|uv|)\\
&=\inf\{M\ |\ m\{(z,t)\ |\ \sigma^{-\delta}|uv|>M\}=0\}\\
&\leq M_{1}M_{2},\ \ \textrm{for any such}\ M_{1}, M_{2}.
\end{split}
\end{equation}
Then (\ref{opw11}) follows. We thus complete the proof.
\end{proof}

The following two estimates, interpolation inequality and Sobolev inequality, follow from the technique of rescaling and applying local estimate:

\begin{prop}[{\bf Interpolation inequality}]
For any $\varepsilon>0$ and any non-negative integer $s$, there is a constant $C=C(Q,p,\delta,s)$ such that
\begin{equation}\label{opw12}
\|u\|_{s+1,p,\delta}\leq\varepsilon\|u\|_{s+2,p,\delta}+\frac{C}{\varepsilon}\|u\|_{0,p,\delta},
\end{equation}
for all $u\in S^{p}_{s+2,\delta}, 1\leq p<\infty$.
\end{prop}
\begin{proof}
The estimates (\ref{opw12}) follows from the technique of rescaling and applying local estimates. By the usual interpolation inequality, we have, for $1\leq p<\infty$, 
\begin{equation}\label{opw121}
\begin{split}
\|u\|_{s+1,p,\delta;A_{R}}&\leq CR^{-\delta}\|u\|_{s+1,p,\delta;A_{1}}\\
&\leq CR^{-\delta}\left(\varepsilon\|u_{R}\|_{s+2,p,\delta;A_{1}}+\frac{C}{\varepsilon}\|u_{R}\|_{0,p,\delta;A_{1}}\right)\\
&\leq \varepsilon\|u\|_{s+2,p,\delta;A_{R}}+\frac{C}{\varepsilon}\|u\|_{0,p,\delta;A_{R}},
\end{split}
\end{equation}
where $C$ is a constant independent of $A_{R}$, for $R\geq 1$.
Therefore, writing $u=\sum_{j=0}^{\infty}u_{j}$ with $u_{0}=u|_{B_{1}}, u_{j}=u|_{A_{2^{j-1}}}$, we see that $(A \preceq B\  \textrm{means}\ A\leq(\textrm{const})B)$
\begin{equation*}
\begin{split}
\|u\|_{s+1,p,\delta}&=\sum_{k=0}^{s+1}\|\nabla^{k}u\|_{p,\delta-k}=\sum_{k=0}^{s+1}\left(\sum_{j=0}^{\infty}\|\nabla^{k}u_{j}\|^{p}_{p,\delta-k}\right)^{1/p}\\
&\preceq\left(\sum_{k=0}^{s+1}\sum_{j=0}^{\infty}\|\nabla^{k}u_{j}\|^{p}_{p,\delta-k}\right)^{1/p}\preceq\left(\sum_{j=0}^{\infty}\|u_{j}\|^{p}_{s+1,p,\delta}\right)^{1/p}\\
&\preceq\left(\varepsilon\sum_{j=0}^{\infty}\|u_{j}\|^{p}_{s+2,p,\delta}+\frac{C}{\varepsilon}\sum_{j=0}^{\infty}\|u_{j}\|^{p}_{0,p,\delta}\right)^{1/p},\ \ (\textrm{by}\ \eqref{opw121}),\\
&\preceq\varepsilon\|u\|_{s+2,p,\delta}+\frac{C}{\varepsilon}\|u\|_{0,p,\delta}.
\end{split}
\end{equation*}
This completes the proof. We remark here that the interpolation inequality also holds when $p=\infty$. But we don't need it in this paper.
\end{proof}

\begin{prop}[{\bf Sobolev inequality}]
If $u\in S^{q}_{k,\delta}$, then we have
\begin{equation}\label{opw13}
\|u\|_{\frac{Qp}{Q-kp},\delta}\leq C\|u\|_{k,q,\delta},
\end{equation}
provided that $Q-kp>0$ and $1<p\leq q\leq\frac{Qp}{Q-kp}$.\\
We also have that
\begin{equation}\label{opw14}
\|u\|_{\infty,\delta}\leq C\|u\|_{k,p,\delta},\ \ \textrm{if}\ Q-kp<0.
\end{equation}
and in fact
\begin{equation}\label{opw15}
|u(z,t)|=o(\rho^{\delta})\ \ \textrm{as}\ \ \rho\rightarrow\infty.
\end{equation}
\end{prop}
\begin{proof}
Suppose that $p^{*}=Qp/(Q-kp)<\infty$, we have
\begin{equation}
\begin{split}
\|u\|_{p^{*},\delta;A_{R}}&=\left(\int_{A_{R}}(\sigma^{-\delta}u)^{p^{*}}\sigma^{-Q}dV_{(z,t)}\right)^{1/p^{*}}\\
&\leq CR^{-\delta}\|u_{R}\|_{p^{*},\delta;A_{1}}\\
&\leq CR^{-\delta}\|u_{R}\|_{k,q,\delta;A_{1}},
\end{split}
\end{equation}
where, for the last inequality, we have used the usual Sobolev inequality applied to $A_{1}$, and by H\"{o}lder's inequality. Rescaling gives
\begin{equation}\label{opw131}
\|u\|_{p^{*},\delta;A_{R}}\leq C\|u\|_{k,q,\delta;A_{R}}.
\end{equation}
Now using the same notation $u_{j}$ as above, we see that
\begin{equation}
\begin{split}
\|u\|_{p^{*},\delta}&=\left(\int_{H_{n}}\|u\|^{p^{*}}\sigma^{-\delta p^{*}-Q}dV_{(z,t)}\right)^{1/p^{*}}\\
&=\left(\sum_{j-0}^{\infty}\int_{A_{2^{j-1}}}|u|^{p^{*}}\sigma^{-\delta p^{*}-Q}dV_{(z,t)}\right)^{1/p^{*}}=\left(\sum_{j-0}^{\infty}\|u\|^{p^{*}}_{p^{*},\delta;A_{2^{j-1}}}\right)^{1/p^{*}}\\
&\leq C\left(\sum_{j-0}^{\infty}\|u\|^{p^{*}}_{k,q,\delta;A_{2^{j-1}}}\right)^{1/p^{*}},\ \ (\textrm{by}\ \eqref{opw131})\\
&\leq C\left(\sum_{j-0}^{\infty}\|u\|^{q}_{k,q,\delta;A_{2^{j-1}}}\right)^{1/q},\ \ q\leq p^{*}\\
&=C\|u\|_{k,q,\delta}.
\end{split}
\end{equation}
Therefore, we obtain \eqref{opw13}. The same rescaling argument implies that
\begin{equation}\label{opw141}
\sup_{A_{R}}|u|\sigma^{-\delta}=\|u\|_{\infty,\delta,A_{R}}\leq C\|u\|_{k,p,\delta;A_{R}},
\end{equation}
which gives \eqref{opw14}. Since $\|u\|_{k,p,\delta}<\infty$, we have
\[\|u\|_{k,p,\delta;A_{R}}\rightarrow 0\ \ \ \textrm{as}\ R\rightarrow\infty,\]
that is,
\[\|u\|_{k,p,\delta;A_{R}}o(1)\ \ \ \textrm{as}\ R\rightarrow\infty,\]
which gives \eqref{opw15}, by means of \eqref{opw141}. We therefore complete the proof.
 
\end{proof}

\subsection{Weighted estimates for sub-elliptic operators}

The rescaling argument and the standard local estimates gives estimates in weighted spaces for sub-elliptic operators whose coefficients are well behaved 
at infinity.

\begin{de}
The operator $u\rightarrow Pu$ defined by
\begin{equation}\label{opw20}
Pu=a^{ij}(z,t)\mathring{e}_{i}\mathring{e}_{j}u+b^{i}(z,t)\mathring{e}_{i}u+c(z,t)u
\end{equation}
will be said to be asymptotic to $\mathring{\Delta}_{b}$ (at rate $\tau$) if there exists $Q<q<\infty$ and $\tau\geq 0$ and constants $C_{1}, \lambda$ such that
\begin{equation}\label{opw21}
\lambda|\zeta|^{2}\leq a^{ij}(z,t)\zeta_{i}\zeta_{j}\leq\lambda^{-1}|\zeta|^{2},\ \ \textrm{for all}\ (z,t)\in H_{n},\ \zeta=\sum_{j=1}^{2n}\zeta_{j}\mathring{e}_{j}\in\xi,
\end{equation}
and
\begin{equation}\label{opw22}
\|a^{ij}-\delta_{ij}\|_{1,q,-\tau}+\|b^{i}\|_{0,q,-1-\tau}+\|c\|_{0,q/2,-2-\tau}\leq  C_{1},
\end{equation}
where $\delta_{ij}$ corresponds to the standard sub-laplacian $\mathring{\Delta}_{b}$ on $H_{n}$.
\end{de}
It is clear that if $P$ is asymptotic to $\mathring{\Delta}_{b}$, then the map $P:S^{p}_{2,\delta}\rightarrow S^{p}_{0,\delta-2}$ is bounded for $1\leq p\leq q$ and $\delta\in\R$. In fact the following weighted estimate holds.

\begin{prop}\label{we01}
Suppose that $P$ is asymptotic to $\mathring{\Delta}_{b}, 1<p\leq q$ and $\delta\in\R$. Then there is a constant $C=C(Q,p,q,s,\delta,C_{1},\lambda)$ such that if $u\in L^{p}_{\delta}$
and $Pu\in S^{p}_{s,\delta-2}$, then $u\in S^{p}_{s+2,\delta}$. Moreover, we have
\begin{equation}\label{opw23}
\|u\|_{s+2,p,\delta}\leq C(\|Pu\|_{s,p,\delta-2}+\|u\|_{p,\delta}).
\end{equation}
\end{prop}
\begin{proof}
Sub-elliptic regularity applies to show that $u\in S^{p}_{s+2,loc}$, and the remaining conclusions follow from the usual interior $L^{p}$ estimates and the rescaling technique.
\end{proof}

Observe that the same argument give the estimate \eqref{opw23} with the $S'^{p}_{k,\delta}$ norms instead. We now investigate the Fredholm properties of $P$. We need the following lemma:

\begin{lem}\label{conbon}
Fix $p\in(1,\infty), p'=p/(p-1)$, and let $a,b\in\R$ be such that $a+b>0$. Suppose that $K'(x,y)$ is the Kernel 
\[K(x,y)=|x|^{-a}|y^{-1}x|^{-Q+a+b}|y|^{-b},\ \ \textrm{for}\ x\neq y,\]
For $u\in L^{p}(H_{n})$ define 
\[Ku(x)=\int_{H_{n}}K(x,y)u(y)dV_{y}.\]
Then there is a constant $c=c(n,p,a,b)$ such that
\[\|Ku\|_{L^{p}}\leq c\|u\|_{L^{p}}\]
if and only if $a<Q/p$ and $b<Q/p'$.
\end{lem}
\begin{proof}
To show that the conditions $a<Q/p$ and $b<Q/p'$ are necessary, suppose that $K$ is a bounded operator on $L^{p}(H_{n})$. We first define the function 
\[v(x)=\int_{|y|\leq 1}K(x,y)dV_{y}.\]
Then we have
\[v(x)=\int_{|y|\leq 1}K(x,y)dV_{y}=\int_{H_{n}}K(x,y)u(y)dV_{y},\]
where 
\[u(y)=\left\{\begin{array}{cc}1&|y|\leq 1\\0&|y|>1\end{array}\right.\]
Since $u(y)\in L^{p}(H_{n})$ and $K$ is a bounded operator on $L^{p}(H_{n})$, we see that $v(x)$ is in $L^{p}(H_{n})$. Next, we study the behavior of $v(x)$ 
at infinity. By the triangle inequality (see \cite{FS}), there exists a constant $\gamma\geq 1$ such that
\[\frac{1}{\gamma}-\frac{|y|}{|x|}\leq\frac{|y^{-1}x|}{|x|}\leq\gamma(1+\frac{|y|}{|x|}).\] 
Therefore, for $|y|\leq 1$, if $|x|$ is large enough, then $K(x,y)\approx\frac{1}{|x|^{Q-b}|y|^{b}}$, and hence $v(x)$ behaves like a constant multiple of $|x|^{-Q+b}$. Then $v(x)$ belongs to $L^{p}(H_{n})$ only if $p(Q-b)>Q$, i.e. only if $b<Q/p'$. On the other hand, we define the function $w(y)$ by
\[w(y)=\int_{|x|\leq 1}K(x,y)dV_{x}.\]
Then, by Fubini's Theorem and H\"{o}lder's inequality, for $u(y)\in L^{p}(H_{n})$, we have
\[\begin{split}
\int_{H_{n}}w(y)u(y)dV_{y}&=\int_{H_{n}}\left(\int_{|x|\leq 1}K(x,y)u(y)dV_{x}\right)dV_{y}\\
&=\int_{|x|\leq 1}\left(\int_{H_{n}}K(x,y)u(y)dV_{y}\right)dV_{x}\\
&\leq\|Ku(x)\|_{L^{p}(B_{1})}\left(\int_{|x|<1}1\ dV_{x}\right)^{1/p'}\\
&\leq\omega_{n}^{1/p'}\|Ku(x)\|_{L^{p}(H_{n})}\leq c\ \omega_{n}^{1/p'}\|u\|_{L^{p}(H_{n})},
\end{split}\]
for some constant $c>0$, where $\omega_{n}$ is the volume of the unit ball $B_{1}$. That is, $w(y)$ is a bounded linear functional on $L^{p}(H_{n})$, and hence belongs to $L^{p}(H_{n})$. As we discuss the behavior of $v(x)$ above, it is easy to see that, for large $|y|$, $w(y)$ behaves like a constant multiple of $|y|^{-Q+a}$. Then $w(y)$ belongs to $L^{p}(H_{n})$ only if $a<Q/p$.

To show that the conditions $a<Q/p$ and $b<Q/p'$ are sufficient, note first that it may be assumed without loss of generality that both $a$ and $b$ are non-negative. (If, say, $a$ is negative, then $b$ is positive, and from the triangle inequality $|x|\leq\gamma(|y|+|y^{-1}x|)$ for some $\gamma\geq 1$, we have 
\[\frac{|x|}{|y^{-1}x|}\leq\gamma\left(1+\frac{|y|}{|y^{-1}x|}\right),\]
which implies that 
\[\left(\frac{|x|}{|y^{-1}x|}\right)^{-a}\leq C\left(1+\frac{|y|}{|y^{-1}x|}\right)^{-a},\]
for some $C>0$. It is seen that
\[K(x,y)\leq\frac{C}{|y^{-1}x|^{Q-b}|y|^{b}}+\frac{C}{|y^{-1}x|^{Q-a-b}|y|^{a+b}}.\]
The boundedness of $K$ in this case then follows immediately from the result in the case in which both $a$ and $b$ are non-nagative.) And it suffices to show that
\begin{equation}\label{be01}
\int_{H_{n}}\int_{H_{n}}K(x,y)u(y)v(x)dV_{y}dV_{x}\leq C\left(\int_{H_{n}}u(y)^{p}dV_{y}\right)^{1/p}\left(\int_{H_{n}}v(x)^{p'}dV_{x}\right)^{1/p'},
\end{equation}
hence \[Ku(x)=\int_{H_{n}}K(x,y)u(y)dV_{y}\in (L^{p'})^{*}(H_{n})=L^{p}(H_{n})\] with operator norm less than $C\left(\int_{H_{n}}u(y)^{p}dV_{y}\right)^{1/p}$.
By H\"{o}lder's inequailty, and noticing that $K(x,y)$ is homogeneous of degree $-Q$, we have 
\begin{equation}\label{be02}
\begin{split}
&\ \ \ \ \int_{H_{n}}\int_{H_{n}}K(x,y)u(y)v(x)dV_{y}dV_{x}\\
&=\int_{H_{n}}\int_{H_{n}}K^{1/p}u(y)\left(\frac{|y|}{|x|}\right)^{Q/pp'}K^{1/p'}v(x)\left(\frac{|x|}{|y|}\right)^{Q/pp'}dV_{y}dV_{x}\\
&\leq \left[\int_{H_{n}}\int_{H_{n}}\left(K^{1/p}u(y)\left(\frac{|y|}{|x|}\right)^{Q/pp'}\right)^{p}\right]^{1/p}\left[\int_{H_{n}}\int_{H_{n}}\left(K^{1/p'}v(x)\left(\frac{|x|}{|y|}\right)^{Q/pp'}\right)^{p'}\right]^{1/p'}\\
&=P^{1/p}Q^{1/p'},
\end{split}
\end{equation}
where
\begin{equation}\label{be03}
\begin{split}
P&=\int_{H_{n}}\int_{H_{n}}\left(K(x,y)u^{p}(y)\left(\frac{|y|}{|x|}\right)^{Q/p'}\right)dV_{y}dV_{x}\\
&=\int_{H_{n}}u^{p}(y)\left(\int_{H_{n}}K(x,y)\left(\frac{|y|}{|x|}\right)^{Q/p'}dV_{x}\right)dV_{y}\\
\end{split}
\end{equation}
and
\begin{equation}\label{be04}
\begin{split}
Q&=\int_{H_{n}}\int_{H_{n}}\left(K(x,y)v^{p'}(x)\left(\frac{|x|}{|y|}\right)^{Q/p}\right)dV_{y}dV_{x}\\
&=\int_{H_{n}}v^{p'}(x)\left(\int_{H_{n}}K(x,y)\left(\frac{|x|}{|y|}\right)^{Q/p}dV_{x}\right)dV_{x}\\
\end{split}
\end{equation}
From \eqref{be02},\eqref{be03},\eqref{be04}, we obtain \eqref{be01}, provided that the integral $\int_{H_{n}}K(x,y)\left(\frac{|y|}{|x|}\right)^{Q/p'}dV_{x}$ has an upper bound which is independent of $y$ and $\int_{H_{n}}K(x,y)\left(\frac{|x|}{|y|}\right)^{Q/p}dV_{y}$ has an upper bound which is independent of $x$. Actually, fixing $y$ and by a change of variables, we have 
\begin{equation}
\begin{split}
\int_{H_{n}}K(x,y)\left(\frac{|y|}{|x|}\right)^{Q/p'}dV_{x}&=\int_{H_{n}}|y|^{-Q}K(\frac{x}{|y|},\frac{y}{|y|})\left(\frac{|y|}{|x|}\right)^{Q/p'}dV_{x}\\
&=\int_{H_{n}}K(z,y^{*})|z|^{-Q/p'}dV_{z},\ z=\frac{x}{|y|},\ \textrm{and}\ y^{*}=\frac{y}{|y|}\\
&=\int_{H_{n}}\frac{1}{|z|^{a+\frac{Q}{p'}}|(y^{*})^{-1}z|^{Q-(a+b)}}dV_{z},
\end{split}
\end{equation}
where
\begin{equation}
\frac{1}{|z|^{a+\frac{Q}{p'}}|(y^{*})^{-1}z|^{Q-(a+b)}}\approx\left\{\begin{array}{ll}
\frac{1}{|z|^{a+\frac{Q}{p'}}}&,x\approx 0\\
\frac{1}{|(y^{*})^{-1}z|^{Q-(a+b)}}&,x\approx 1\\
\frac{1}{|z|^{Q+\frac{Q}{p'}-b}}&,x\approx \infty
\end{array}\right.
\end{equation}
Since $a<\frac{Q}{p}$ and $b<\frac{Q}{p'}$, we have that $0<a+\frac{Q}{p'}<Q,\ 0<Q-(a+b)<Q$ and $Q+\frac{Q}{p'}-b>Q$, thus the integral $\int_{H_{n}}K(x,y)\left(\frac{|y|}{|x|}\right)^{Q/p'}dV_{x}$ converges, depending on $y^{*}$ which belongs to the unit sphere. Therefore, it has an upper bound which is independent of $y$. Similarly, $\int_{H_{n}}K(x,y)\left(\frac{|x|}{|y|}\right)^{Q/p}dV_{y}$ has an upper bound which is independent of $x$. Therefore, we have completed the proof.
\end{proof}

\begin{thm}\label{keythm}
Suppose that $\delta<0$, $1<p<\infty$, and $s$ is a non-negative integer. Then the map
\begin{equation}\label{opw24}
\mathring{\Delta}_{b}:S'^{p}_{s+2,\delta}\rightarrow S'^{p}_{s,\delta-2}
\end{equation}
is an isomorphism and there is a constant $C=C(Q,p,\delta,s)$ such that
\begin{equation}\label{opw25}
\|u\|'_{s+2,p,\delta}\leq C\|\mathring{\Delta}_{b}u\|'_{s,p,\delta-2}.
\end{equation}
\end{thm}
\begin{proof}
It suffices to prove \eqref{opw25} for $s=0$. We first assume that $-2n<\delta<0$ and show that the distribution inverse has convolution kernel $K_{0}(x,y)$:
\begin{equation}\label{invop}
c_{0}K_{0}(x,y)=c_{0}\Phi_{0}(y^{-1}x)=|y^{-1}x|^{-2n}.
\end{equation}
Let us first show that \eqref{invop} defines a bounded operator from $S'^{p}_{0,\delta-2}$ to $S'^{p}_{0,\delta}$. Lemma \ref{conbon} shows that the kernel 
\[K(x,y)=|x|^{-\delta-\frac{2n+2}{p}}K_{0}(x,y)|y|^{\delta-2+\frac{2n+2}{p}}\]
defines a bounded operator $L^{p}\rightarrow L^{p}$ when $-2n<\delta< 0$. Then we have
\begin{equation}
\begin{split}
\|K_{0}u\|'_{p,\delta}&=\left(\int_{H_{n}\setminus\{0\}}|K_{0}u(x)|^{p}|x|^{-\delta p-Q}dV_{x}\right)^{1/p}\\
&=\left(\int_{H_{n}\setminus\{0\}}\left|\int_{H_{n}\setminus\{0\}}|x|^{-\delta -\frac{Q}{p}}K_{0}(x,y)u(y)dV_{y}\right|^{p}dV_{x}\right)^{1/p}\\
&=\left(\int_{H_{n}\setminus\{0\}}\left|(KU)(x)\right|^{p}dV_{x} \right)^{1/p},\ \ \textrm{where}\ U(y)=u(y)|y|^{-(\delta-2)-\frac{Q}{p}}\\
&=\|KU\|_{L^{p}}\leq c\|U\|_{L^{p}}\ \ (\textrm{by Lemma}\ \ref{conbon})\\
&=c\|u\|'_{p,\delta-2}.
\end{split}
\end{equation}
Thus we have shown that $K_{0}:S'^{p}_{0,\delta-2}\rightarrow S'^{p}_{0,\delta}$ is bounded when $-2n<\delta< 0$. 
Recall that we have $K_{0}\mathring{\Delta}_{b}u=u$, for all $u\in C^{\infty}_{0}(H_{n}\setminus\{0\})$, so the boundness of $K_{0}$
gives 
\[\|u\|'_{p,\delta}\leq C\|\mathring{\Delta}u\|'_{p,\delta-2},\ \ \ \textrm{for all}\ u\in S'^{p}_{0,\delta},\]
since $C^{\infty}_{0}(H_{n}\setminus\{0\})$ is dense. This and estimate \eqref{opw23} yield \eqref{opw25}.

Now suppose that $\{u_{j}\}\subset S'^{p}_{2,\delta}, \{f_{j}\}\subset S'^{p}_{0,\delta-2}$ are sequences such that 
$f_{j}=\mathring{\Delta}u_{j}$ and $f_{j}\rightarrow f$.  By \eqref{opw25}, $\{u_{j}\}$ is a Cauchy sequence and hence is convergent 
to $u\in S'^{p}_{2,\delta}$ with $\mathring{\Delta}u=f$, since $\mathring{\Delta}:S'^{p}_{2,\delta}\rightarrow S'^{p}_{0,\delta-2}$ is bounded.
This map thus has closed range. On the other hand, by \eqref{opw25}, the kernel of this map is trivial. Also, by \eqref{opw25}, we have 
$K_{0}f\in S'^{p}_{2,\delta}$, and hence $\mathring{\Delta}(K_{0}f)(x)=f$, for all $f\in C_{0}^{\infty}(H_{n}\setminus\{0\})$, so it is also surjective 
and this establishes the isomorphism for the cases $-2n<\delta<0$. For the other cases $\delta\leq -2n$, we need the following commutative diagram
\begin{equation*}
\begin{array}{cccc}
\mathring{\Delta}_{b}:&S'^{p}_{s+2,\delta_{1}}&\longrightarrow&S'^{p}_{s,\delta_{1}-2}\\
&\uparrow&&\uparrow\\
\mathring{\Delta}_{b}:&S'^{p}_{s+2,\delta_{2}}&\longrightarrow&S'^{p}_{s,\delta_{2}-2}
\end{array}
\end{equation*}
where $\delta_{2}<\delta_{1}$ and the vertical maps are the inclusions, which, by Proposition \ref{ebe1}, are continuous. Suppose that $\mathring{\Delta}_{b}:S'^{p}_{s+2,\delta_{1}}\longrightarrow S'^{p}_{s,\delta_{1}-2}$ is isomorphic, then from the above diagram, it is easy to see that the map $\mathring{\Delta}_{b}:S'^{p}_{s+2,\delta_{2}}\longrightarrow S'^{p}_{s,\delta_{2}-2}$ is one-to-one. Since $\mathring{\Delta}_{b}$ is self-adjoint, the map $\mathring{\Delta}_{b}:S'^{p}_{s+2,\delta_{2}}\longrightarrow S'^{p}_{s,\delta_{2}-2}$ is onto, and hence isomorphic. We claim that \eqref{opw25} holds for $\delta_{2}$.
For a contradiction, we assume that \eqref{opw25} does not hold for $\delta_{2}$. Then there exists a sequence of $u_{j}\in S'^{p}_{s+2,\delta_{2}}$ such that 
$\|u_{j}\|'_{s+2,p,\delta_{2}}=1$ and 
\begin{equation}
\|\mathring{\Delta}_{b}u_{j}\|'_{s,p,\delta_{2}-2}<\frac{1}{j}.
\end{equation}
Thus there exists a subsequence $\{u_{j}\}$ converges weakly to $u_{\infty}$ such that $\|u_{\infty}\|'_{s+2,p,\delta_{2}}=1$ and $\|\mathring{\Delta}_{b}u_{\infty}\|'_{s,p,\delta_{2}-2}=0$, which implies that $u_{\infty}=0$, due to that $\mathring{\Delta}_{b}$ is one-to-one. We therefore get a contradiction, and thus have \eqref{opw25} for all $\delta<0$.
\end{proof}

The scale-broken estimate \eqref{opw26} below is the key to proving Fredholm properties.

\begin{thm}\label{kefp}
Suppose that $P$ is asymptotic to $\mathring{\Delta}_{b}$ at rate $\tau>0$ and $Q<q<\infty$, and $\delta<0$, $s$ is a non-negative integer. 
Then, for $1<p\leq q$, the map
\[P:S^{p}_{s+2,\delta}\rightarrow S^{p}_{s,\delta-2}\]
has finite-dimensional kernel and closed range. In addition, for any $u\in S^{p}_{s+2,\delta}$, we have 
\begin{equation}\label{opw26}
\|u\|_{s+2,p,\delta}\leq C(\|Pu\|_{s,p,\delta-2}+\|u\|_{L^{p}(B_{R})}),
\end{equation}
where $C,R$ are constants depending on $P,\delta,Q,p,q$.
\end{thm}
\begin{proof}
It suffices to prove \eqref{opw26} for $s=0$. Define the operator norm
\[\|P-\mathring{\Delta}_{b}\|_{op}=\sup\{\|(P-\mathring{\Delta}_{b})u\|_{p,\delta-2}\ :\ u\in S^{p}_{2,\delta},\ \|u\|_{2,p,\delta}=1\}\]
and let $\|\cdot\|_{op,R}$ denote the same norm restricted to function with support in $E_{R}=H_{n}\setminus B_{R}$.
For \eqref{opw26}, first, we claim that 
\begin{equation}\label{opw27}
\|P-\mathring{\Delta}_{b}\|_{op,R}=o(1),\ \ \textrm{as}\ R\rightarrow\infty.
\end{equation}
If $\supp{u}\subset E_{R}$, since $Q<q$, 
\begin{equation} \label{opw28}
\begin{split}
&\|(P-\mathring{\Delta}_{b})u\|_{p,\delta-2}\\
\leq &\|(a^{ij}(z,t)-\delta_{ij})\partial_{ij}u+b^{i}(z,t)\partial_{i}u+c(z,t)u\|_{p,\delta-2}\\
\leq &\|(a^{ij}(z,t)-\delta_{ij})\partial_{ij}u\|_{p,\delta-2;E_{R}}+\|b^{i}(z,t)\partial_{i}u\|_{p,\delta-2;E_{R}}+\|c(z,t)u\|_{p,\delta-2;E_{R}},
\end{split}
\end{equation}
where
\begin{equation} \label{opw29}
\|(a^{ij}(z,t)-\delta_{ij})\partial_{ij}u\|_{p,\delta-2;E_{R}}
\leq \sup_{|(z,t)|>R}\left\{|a^{ij}(z,t)-\delta_{ij}|\right\}\|\partial_{ij}u\|_{p,\delta-2;E_{R}};
\end{equation}

\begin{equation} \label{opw30}
\begin{split}
\|b^{i}(z,t)\partial_{i}u\|_{p,\delta-2;E_{R}}&\leq \|b^{i}\|_{Q,-1;E_{R}}\|\partial_{i}u\|_{\frac{Qp}{Q-p},\delta-1},\ \ \textrm{by}\ \eqref{opw11}\\
&\leq C\|b^{i}\|_{Q,-1;E_{R}}\|\partial_{i}u\|_{1,p,\delta-1},\ \ \textrm{by}\ \eqref{opw13}\\
&\leq C\|b^{i}\|_{q,-1-\tau;E_{R}}\|\partial_{i}u\|_{1,p,\delta-1},\ \ \textrm{by}\ \eqref{opw10};
\end{split}
\end{equation}
and, similarly, we have
\begin{equation} \label{opw31}
\|c(z,t)u\|_{p,\delta-2;E_{R}}\leq C\|c\|_{q/2,-2-\tau;E_{R}}\|u\|_{2,p,\delta}.
\end{equation}
Substituting \eqref{opw29},\eqref{opw30} and \eqref{opw31} into \eqref{opw28}, we obtain 
\begin{equation*}
\begin{split}
\|(P-\mathring{\Delta}_{b})u\|_{p,\delta-2}&\leq \sup_{|(z,t)|>R}\left\{|a^{ij}(z,t)-\delta_{ij}|\right\}\|\partial_{ij}u\|_{p,\delta-2;E_{R}}\\
&+ C\|b^{i}\|_{q,-1-\tau;E_{R}}\|\partial_{i}u\|_{1,p,\delta-1}+C\|c\|_{q/2,-2-\tau;E_{R}}\|u\|_{2,p,\delta},
\end{split}
\end{equation*}
and hence \eqref{opw27} holds.

Let $\chi\in C_{0}^{\infty}(B_{2})$ be a patch function, $0\leq\chi\leq 1,\chi=1$ in $B_{1}$, and set $\chi_{R}(x)=\chi(R^{-1}x)$. Writting
$u=u_{0}+u_{\infty},\ u_{0}=\chi_{R}u,u_{\infty}=(1-\chi_{R})u$ with $R$ a constant to be determined, the sharp estimate \eqref{opw25} yields 
\begin{equation*}
\begin{split}
\|u_{\infty}\|_{2,p,\delta}&\leq C\|\mathring{\Delta}_{b}u_{\infty}\|_{0,p,\delta-2}\\
&\leq C(\|Pu_{\infty}\|_{0,p,\delta-2}+\|(P-\mathring{\Delta}_{b})u_{\infty}\|_{0,p,\delta-2})\\
&\leq C(\|Pu_{\infty}\|_{0,p,\delta-2}+\|P-\mathring{\Delta}_{b}\|_{op,R}\|u_{\infty}\|_{2,p,\delta}),
\end{split}
\end{equation*}  
in which the second term of the right hand side can be absorbed into the left hand side, for $R$ sufficiently large, and 
we estimate
\begin{equation*}
\begin{split}
\|Pu_{\infty}\|_{0,p,\delta-2}&=\|Pu-Pu_{0}\|_{0,p,\delta-2}\\
&\leq \|Pu\|_{0,p,\delta-2}+\|2a^{ij}(\partial_{i}u)(\partial_{j}\chi_{R})+(a^{ij}\partial_{ij}\chi_{R}+b^{i}\partial_{i}\chi_{R})u\|_{0,p,\delta-2;A_{R}}\\
&\leq \|Pu\|_{0,p,\delta-2}+C\|u\|_{1,p,\delta;A_{R}}.
\end{split}
\end{equation*}  
We therefore get
\[\|u_{\infty}\|_{2,p,\delta}\leq C(\|Pu\|_{0,p,\delta-2}+\|u\|_{1,p,\delta;A_{R}},\]
for $R$ large enough. Using this \eqref{opw23} and the interpolation inequality \eqref{opw12} gives \eqref{opw26}.
Now suppose that $\{u_{j}\}$ is a sequence in Ker$P$ satisfying $\|u_{j}\|_{2,p,\delta}=1$, so that by the Rellich lemma 
we may assume, without loss of generality, that $\{u_{j}\}$ converges strongly in $L^{p}(B_{R})$. Estimate \eqref{opw26} now shows that 
 $\{u_{j}\}$ is a Cauchy sequence and hence convergent in $S^{p}_{2,\delta}$ which implies that Ker$P$ is finite-dimensional. To show that
 $P$ has closed range, as the argument in Theorem \ref{keythm}, it suffices to show that there is a constant $C$ such that 
 \[\|u\|_{2,p,\delta}\leq C\|Pu\|_{0,p,\delta-2},\ \ \ \textrm{for all}\ u\in Z,\]
 where $Z$ is a closed subspace such that $S^{p}_{2,\delta}=Z+\textrm{Ker}P$. For if this were not the case, there would be a sequence 
 $\{u_{j}\}\subset Z$ such that $\|u_{j}\|_{2,p,\delta}=1$ and $\|Pu\|_{0,p,\delta-2}\rightarrow 0$. The usual Rellich lemma applied to 
 \eqref{opw26} shows that $\{u_{j}\}$ has a subsequence which is Cauchy in $Z$ and whose limit is a non-zero element of Ker$P\cap Z$, which is 
 a contradiction. Therefore, $P$ has closed range.
\end{proof}
We are interested in the dimension of the kernel of $P$, which will be denoted by
\begin{equation}
N(P,\delta)=\textrm{dim\ ker}(P:S^{p}_{2,\delta}\rightarrow S^{p}_{0,\delta-2})
\end{equation}
with $1<p\leq q$, where $Q=2n+2<q<\infty$.  

\begin{prop}\label{mainres}
Suppose that $P$ and its formal adjoint $P^{*}$ both satisfy conditions \eqref{opw21} and \eqref{opw22} with $Q<q<\infty$, and $-2n<\delta<0,\ \frac{q}{q-1}<p\leq q$. Then $P:S^{p}_{2,\delta}\rightarrow S^{p}_{0,\delta-2}$ is a Fredholm operator with Fredholm index $\iota(P,\delta)=N(P,\delta)-N(P^{*},2-Q-\delta)$.
\end{prop}
\begin{proof}
Suppose that the formal adjoint
\[P^{*}:(S^{p}_{0,\delta-2})^{*}\rightarrow(S^{p}_{2,\delta})^{*}\]
of $P$ is also asymptotic to $\mathring{\Delta}_{b}$ which satisfies the condition \eqref{opw21} and \eqref{opw22}. Here $(S^{p}_{2,\delta})^{*}$ is the subspace of $D'(H_{n})$ consisting of those distributions which extend to give bounded linear functionals on $S^{p}_{2,\delta}$, endowed with the dual norm.
From the H\"{o}lder inequality  \eqref{opw11}, we have 
\[\|uv\|_{1,-Q}\leq\|u\|_{p,\delta-2}\|v\|_{p',2-Q-\delta},\] 
which implies that $:(S^{p}_{0,\delta-2})^{*}=S^{p'}_{2-Q-\delta}$, where $p'=p/(p-1)$. If $-2n<\delta<0$, which implies that $-2n<2-Q-\delta<0$, and moreover
$\frac{q}{q-1}<p\leq q$ means that $1<p'\leq q$, then Proposition \ref{we01} shows that $Ker(P^{*})\subset S^{p'}_{2,2-Q-\delta}$ and hence 
\[\textrm{dim\ coker}P=\textrm{dim\ ker}P^{*}=N(P^{*},2-Q-\delta).\]
This is finite by Theorem \ref{kefp}; hence $P$ is Fredholm with Fredhom index
\begin{equation}
\iota(P,\delta)=N(P,\delta)-N(P^{*},2-Q-\delta).
\end{equation}
\end{proof}

Let $\xi_{0}=\ker{\Theta}$ be the standard contact bundle on the Heisenberg group. Instead of the standard pseudo-hermitian structure $(J_{0},\Theta)$, we
consider another pseudo-hermitian structure $(J,\theta)$ with $\ker{\theta}=\xi_{0}$. Let $P=\Delta_{b}$ be the corresponding sub-Laplacian, then $P^{*}=P$. We write 
\[\Delta_{b}=g_{ij}(z,t)\mathring{e}_{i}\mathring{e}_{j}+\cdots.\]
We have the following corollary.
\begin{cor}\label{mainres01}
Suppose that $g_{ij}$ is uniformly sub-elliptic in $(\R^{2n+1},\xi_{0})$ and $(g_{ij}-\delta_{ij})\in S^{q}_{1,-\tau}$ for some $Q<q<\infty$, and that $-2n<\delta<0$ and $\frac{q}{q-1}<p\leq q$. Then $\Delta_{b}:S^{p}_{s+2,\delta}\rightarrow S^{p}_{s,\delta-2}$ is an isomorphism.
\end{cor}
\begin{proof}
From Proposition \ref{mainres}, the sub-Laplacian $\Delta_{b}:S^{p}_{2,\delta}\rightarrow S^{p}_{0,\delta-2}$ is Fredholm. And since $-2n<\delta<0$ if and only if $-2n<\delta^{*}<0$, where $\delta^{*}=2-Q-\delta$. It suffices to show that $N(\Delta_{b},\delta)=0$. But if $\Delta_{b}u=0$ and $u\in S^{p}_{2,\delta}$, then $u=o(1)$ at infinity and {\bf the strong maximum principle} for weak solutions (see \cite{GT}, Theorem 8.19) shows that $u=0$.
\end{proof}

\section{Asymptotically Flat Pseudo-hermitian Manifolds}

\begin{de}
A $(2n+1)$-dimensional pseudo-hermitian manifold $(N,J,\theta)$ is said to be asymptotically flat pseudo-hermitian if $N=N_{0}\cup N_{\infty}$, with $N_{0}$ compact and $N_{\infty}$ diffemorphic to $H_{n}\setminus B_{\rho_{0}}$ in which $(J,\theta)$ is close to $(\mathring{J},\mathring{\theta})$ in the sense that
\begin{equation}
\begin{split}
\theta&=(1+c_{n}A\rho^{-2n}+O(\rho^{-2n-1}))\mathring{\theta}+O(\rho^{-2n-1})_{\beta}dz^{\beta}+O(\rho^{-2n-1})_{\bar\beta}dz^{\bar\beta};\\
\theta^{\alpha}&=O(\rho^{-2n-1})\mathring{\theta}+O(\rho^{-2n-2})_{\bar\beta}{}^{\alpha}dz^{\bar\beta}+(1+\tilde{c}_{n}A\rho^{-2n}+O(\rho^{-2n-1}))\sqrt{2}dz^{\alpha},
\end{split}
\end{equation}
for some $A\in\R$ and a unitary co-frame $\theta^{\alpha}$ in coordinates $(z^{\beta},z^{\bar\beta},t)$ (called asymptotic coordinates) for $N_{\infty}$ on which $\rho=((\sum_{\beta=1}^{n}|z^{\beta}|^{2})^{2}+t^{4})^{1/4}$ is defined. We also require the Tanaka-Webster scalar curvature $W\in L^{1}(N)$.
\end{de}

Suppose that $(M,J)$ is a {\bf spherical} in a neighborhood of a point $x\in M$. If we blow up at $x$ through the Green's function of the CR invariant sub-Laplacian
\[G_{x}=\frac{a_{n}}{2\pi}\rho^{-2n}+A_{x}+O(\rho),\] 
which is expressed in terms of the CR normal coordinates, then we obtain an asymptotically flat pseudo-hermitian manifold (see \cite{CC} for more explanation). We denote such a manifold by $N(x)$. In such a case, we have
\begin{equation}
\begin{split}
\theta&=(1+c_{n}A_{x}\rho^{-2n}+O(\rho^{-2n-1}))\mathring{\theta};\\
\theta^{\alpha}&=O(\rho^{-2n-1})\mathring{\theta}+(1+\tilde{c}_{n}A_{x}\rho^{-2n}+O(\rho^{-2n-1}))\sqrt{2}dz^{\alpha},
\end{split}
\end{equation}
where $c_{n}=\frac{4\pi}{na_{n}}$ and $\tilde{c}_{n}=\frac{2\pi}{na_{n}}$. Let $\{Z_{\alpha}, Z_{\bar\alpha},T\}$ be the frame dual to $\{\theta^{\alpha}, \theta^{\bar\alpha},\theta\}$. It is easy to see that 
\begin{equation}
Z_{\alpha}=(1-\tilde{c}_{n}A_{x}\rho^{-2n}+O(\rho^{-2n-1}))\mathring{Z}_{\alpha}+O(\rho^{-2n-2})_{\alpha}{}^{\bar\beta}\mathring{Z}_{\bar\beta}.
\end{equation}
On the other hand, from the structure equations of pseudo-hermitian manifolds, we have the asymptotic behavior of the pseudo-hermitian connection forms
\begin{equation}
\theta_{\alpha}{}^{\beta}=O(\rho^{-2n-2})\mathring{\theta}+B_{\alpha}{}^{\beta}{}_{\gamma}dz^{\gamma}+C_{\alpha}{}^{\beta}{}_{\bar\gamma}dz^{\bar\gamma},
\end{equation}
where for any fixed $\alpha$, we have
\begin{equation}
\begin{split}
C_{\alpha}{}^{\beta}{}_{\bar\gamma}&=O(\rho^{-2n-2}),\ \textrm{for}\ \beta\neq\alpha, \gamma\neq\alpha\\
C_{\alpha}{}^{\beta}{}_{\bar\alpha}&=\frac{-inc_{n}A_{x}z^{\beta}\omega}{\rho^{2n+4}}+O(\rho^{-2n-2}),\ \textrm{for}\ \beta\neq\alpha\\
C_{\alpha}{}^{\alpha}{}_{\bar\gamma}&=\frac{-in\tilde{c}_{n}A_{x}z^{\gamma}\omega}{\rho^{2n+4}}+O(\rho^{-2n-2}),\ \textrm{for}\ \gamma\neq\alpha\\
C_{\alpha}{}^{\alpha}{}_{\bar\alpha}&=\frac{-in(c_{n}+\tilde{c}_{n})A_{x}z^{\alpha}\omega}{\rho^{2n+4}}+O(\rho^{-2n-2}),\\
\end{split}
\end{equation}
and $B_{\alpha}{}^{\beta}{}_{\gamma}=-\overline{C_{\beta}{}^{\alpha}{}_{\bar\gamma}}$. Therefore, we have
\begin{equation}\label{asybesub}
\Delta_{b}=\left(1+\frac{4\tilde{c}_{n}A_{x}}{\rho^{2n}}\right)\mathring{\Delta}_{b}+\textrm{higher decay order},
\end{equation}
which implies that $\Delta_{b}$ is asymptotic to $\mathring{\Delta}_{b}$ for some $q$ with $Q<q<\infty$. We remark that if the manifold $(M,J)$ is not spherical, then, in general, formula \eqref{asybesub} does not hold.

As used in Corollary \ref{mainres01}, the strong maximum principle implies the following theorem.

\begin{thm}\label{mainres02} 
Suppose that $(M,J)$ is a CR manifold which is {\bf spherical} in a neighborhood of a point $x\in M$. We consider the blow up manifold $N(x)$ at this point $x$. Choose $q>2n+2$ such that $\Delta_{b}$ is asymptotic to $\mathring{\Delta}_{b}$, $\frac{q}{q-1}<p\leq q$ and $-2n<\delta<0$. Then
\begin{equation}
\Delta_{b}:S^{p}_{s+2,\delta}(N(x))\rightarrow S^{p}_{s,\delta-2}(N(x))
\end{equation}
is an isomorphism.
\end{thm}

\subsection{The Weitzenbock formula for $n=2$}
In this subsection, we will give an isomorphism theorem for Dirac operators on spinor bundles. More details and notations on the geometry of context of spinors here may be found in \cite{CC}. Suppose $(M,J,\theta)$ is a spherical CR manifold with a spin structure $Spin(\xi)$ on the contact bundle $\xi=\ker{\theta}$, together with the Levi metric $L_{\theta}$. Let $\Lambda_{R}^{k}(n)$ and $\Lambda_{C}^{k}(n)$ denote the real and complex vector spaces respectively, spanned by $\{\omega^{j_1}\wedge\cdots\wedge\omega^{j_k}\ |\ 1\leq j_{1}<\cdots<j_{k}\leq n\}$ (view symbols $\omega^{1},\cdots,\omega^{n}$ as independent vectors).
Recall, from \cite{CC}, that there is a canonical representation $\mathcal{N}$ of the Clifford algebra $C_{2n}(-1)$ such that $\Lambda_{C}^{*}(n)$ is an irreducible $C_{2n}(-1)$-modules, and $\Lambda_{C}^{even}(n), \Lambda_{C}^{odd}(n)$ are two irreducible $Spin(2n)$-modules. We hence have the following three associated vector bundles:
\[S=Spin(\xi)\times_{\mathcal{N}}\Lambda_{C}^{*}(n),\ S^{+}=Spin(\xi)\times_{\mathcal{N}}\Lambda_{C}^{even}(n),\ S^{-}=Spin(\xi)\times_{\mathcal{N}}\Lambda_{C}^{odd}(n),\] 
which are called spinor bundles. On each spinor bundle there is associated a unique spin connection $\nabla$, constructed from the pseudo-hermitian connection, and hence a Dirac operator, which is denoted by $D_{\xi}$. We have a Weitzenbock-type formula as follows:
\[D_{\xi}^{2}=\nabla^{*}\nabla+W-2\sum_{\beta=1}^{n}e_{\beta}e_{n+\beta}\nabla_{T},\]
where $\{e_{1},\cdots,e_{2n}\}$ is a local orthonormal frame field such that $e_{n+\beta}=Je_{\beta}$ for $1\leq\beta\leq n$. Fortunately, due to some algebraic property in the case $n=2$, we reduce the formula to
\begin{equation}\label{wtfby3.3}
D_{\xi}^{2}=\nabla^{*}\nabla+W,\ \ \ \textrm{on}\ S^{+}
\end{equation}
It is this formula that makes Witten's approach work on this geometry. Actually we have the following isomorphism theorem for Dirac operators.

\begin{thm}\label{mainres03} 
Suppose $n=2$ and suppose that $(M,J)$ is a CR manifold which is {\bf spherical} in a neighborhood of a point $x\in M$. We consider the blow up manifold $N(x)$ at this point $x$. Suppose that $\frac{q}{q-1}<p\leq q\ (q$ is specified as in Theorem \ref{mainres02}$)$ and $-2n<\delta<0$. Then
\begin{equation}
D^{2}_{\xi}:S^{p}_{s+2,\delta}(S^{+})\rightarrow S^{p}_{s,\delta-2}(S^{+})
\end{equation}
is an isomorphism.
\end{thm}
\begin{proof}
 Let $\varphi=\varphi^{\lambda}s_{\lambda}$ be a spinor spanned by a frame field of spinors $\{s_{\lambda}\}$. We compute
 \begin{equation}
 \begin{split}
 \nabla^{*}\nabla\varphi&=-\sum_{a=1}^{2n}\nabla^{2}_{e_{a},e_{a}}\varphi\\
 &=-\sum_{a=1}^{2n}\left(\nabla_{e_a}\nabla_{e_a}\varphi-\nabla_{\nabla_{e_a}e_{a}}\varphi\right)\\
 &=-\sum_{a=1}^{2n}(e_{a}e_{a}\varphi^{\lambda}-(\nabla_{e_a}e_{a})\varphi^{\lambda})s_{\lambda}+2(e_{a}\varphi^{\lambda})\nabla_{e_a}s_{\lambda}+\varphi^{\lambda}\left(\nabla_{e_a}\nabla_{e_a}s_{\lambda}-\nabla_{\nabla_{e_a}e_{a}}s_{\lambda}\right)\\
 &=-(\Delta_{b}\varphi^{\lambda})s_{\lambda}+\ \textrm{l.o.t.},
 \end{split}
 \end{equation}
where \begin{equation*}
\begin{split}
\Delta_{b}f&=tr(\nabla df)=e_{a}e_{a}f-(\nabla_{e_a}e_{a})f,\ \ \textrm{for a function}\ f\\
\textrm{l.o.t.}&=2(e_{a}\varphi^{\lambda})\nabla_{e_a}s_{\lambda}+\varphi^{\lambda}\left(\nabla_{e_a}\nabla_{e_a}s_{\lambda}-\nabla_{\nabla_{e_a}e_{a}}s_{\lambda}\right).
\end{split}
\end{equation*}
Therefore the principal part of each component of $D_{\xi}^2$ is the sub-Laplacian $\Delta_{b}$, which is asymptotic to $\mathring{\Delta}_{b}$ at rate $\tau>0$.
Therefore, $D_{\xi}^2$ satisfies the scale-broken estimate \eqref{opw26}, and hence is Fredholm with adjoint
\[D^{2}_{\xi}=D_{\xi}^{*2}:S^{p'}_{s+2,2-Q-\delta}(S^{+})\rightarrow S^{p'}_{s,-Q-\delta}(S^{+}).\]
On the other hand, from \eqref{opw26} again, we see that $\ker{(D_{\xi}^2,\delta)}\subset C^{\infty}(N(x))$, thus if $\varphi\in\ker{(D_{\xi}^2,\delta)}$, then $|\varphi|^{2}\rightarrow 0$ at infinity and from \eqref{wtfby3.3} we have
\[\frac{1}{2}\Delta_{b}|\varphi|^{2}=|\nabla_{b}\varphi|^{2}+W|\varphi|^{2}\geq 0.\]
The strong maximum principle implies that $|\varphi|^{2}=0$. Since that $-2n<\delta<0$ if and only if $-2n<2-Q-\delta<0$, this shows that for $-2n<\delta<0$, both $\ker{D_{\xi}^2}$ and $\ker{D_{\xi}^{*2}}$ are trivial, and hence $D_{\xi}^2$ is an isomorphism.
\end{proof}

\bibliography{main}

\begin{thebibliography}{99}
\bibitem{CC} J.-H. Cheng and H.-L. Chiu, {\it Positive Mass Theorem and the CR Yamabe Equation on Five-dimensional Contact Spin Manifolds}, arXiv:2107.04209v1.

\bibitem{CCY} J.-H. Cheng, H.-L. Chiu and P. Yang, {\it Uniformization of spherical CR manifolds}, Adv. Math. 255(2014), 182-216.

\bibitem{CMY} J.-H. Cheng, A. Malchiodi and P. Yang, {\it A positive mass theorem in three dimensional Cauchy-Riemann geometry}, Adv. Math. 308(2017), 276-347.

\bibitem{Bar} R. Bartnik, {\it The Mass of an Asymptotically Flat Manifold}, Comm. Pure and App. Math., (1986), 661-693.

\bibitem{FS} G. B. Folland and E. M. Stein, {\it Estimates for the $\bar{\partial}_{b}$ Complex Analysis on the Heisenberg group}, Comm. Pure and App. Math., (1974), 429-522.

\bibitem{GT} D. Gilbarg and N. Trudinger, {\it Elliptic Partial Differential Equations of Second Order}, 2nd edition, Springer- Verlag, (1983).

\bibitem{PT} T. Parker and C. Taubes, {\it On Witten's proof of the positive energy theorem}, Comm. Math. Phys. 84(1982), 223-238.

\bibitem{W} E. Witten, {it A simple proof of the positive energy theorem}, Comm. Math. Phys. 80(1981), 381-402.



\end{thebibliography}
\bibliographystyle{plain}

\end{document}